%% file: 2022MFCSskolem.tex
\documentclass[a4paper,UKenglish,cleveref, autoref, thm-restate, final]{lipics-v2021}

\pdfoutput=1 
\hideLIPIcs  

\title{On the Skolem Problem for Reversible Sequences}

\titlerunning{On the Skolem Problem for Reversible Sequences}

\author{George Kenison}{Institute of Logic and Computation, TU Wien, Vienna, Austria}{george.kenison@tuwien.ac.at}{}{}

\authorrunning{G. Kenison} 

\Copyright{George Kenison} 



\nolinenumbers 

\input{preamble}

\usepackage{xparse}
\usepackage{enumitem}
\usepackage{etoolbox}
\usepackage{letltxmacro}

\makeatletter
\AtBeginDocument{%
\def\strip@@parentheses(#1){#1}
\LetLtxMacro\enumerate@@item\item
\AtBeginEnvironment{enumerate}{%
  \RenewDocumentCommand{\item}{o}{%
    \IfValueTF{#1}{
      \enumerate@@item[#1]%
      \protected@edef\@currentlabel{\strip@@parentheses#1}
    }{%
      \enumerate@@item
    }%
  }%
}
}
\makeatother

\EventEditors{Stefan Szeider, Robert Ganian, and Alexandra Silva}
\EventNoEds{3}
\EventLongTitle{47th International Symposium on Mathematical Foundations of Computer Science (MFCS 2022)}
\EventShortTitle{MFCS 2022}
\EventAcronym{MFCS}
\EventYear{2022}
\EventDate{August 22--26, 2022}
\EventLocation{Vienna, Austria}
\EventLogo{}
\SeriesVolume{241}
\ArticleNo{7}

\begin{document}

\maketitle

\begin{abstract}
Given an integer linear recurrence sequence \(\seq[\infty]{X_n}{n=0}\), the Skolem Problem asks to determine whether there is an \(n\in\N_0\) such that \(X_n = 0\).
Recent work by Lipton, Luca, Nieuwveld, Ouaknine, Purser, and Worrell proved that the Skolem Problem is decidable for a class of {reversible} sequences of order at most seven.
Here we give an alternative proof of their result.
Our novel approach employs a powerful result for Galois conjugates that lie on two concentric circles due to Dubickas and Smyth.
\end{abstract}

\input{introduction.tex}
\input{preliminaries.tex}
\input{decidability.tex}
\input{octics.tex}

\input{future.tex}

\newpage


\bibliography{Skolembib}



\end{document}

%% file: preamble.tex
\usepackage[usenames,dvipsnames,table]{xcolor}

\usepackage{amsmath, amssymb, mathrsfs} 
\usepackage{mathtools}
\usepackage{url}
\usepackage{mleftright}
\usepackage{microtype}
\usepackage[small,basic]{complexity}
\usepackage{tikz}
\usetikzlibrary{arrows, positioning}
\usepackage{aliascnt}

\newcommand{\vect}[1]{\underline{#1}}

\renewcommand{\epsilon}{\ensuremath\varepsilon}

\newclass{\PTIME}{PTIME}

\newcommand{\obeta}{\smash{\overline{\beta}}\vphantom{\beta}}
\newcommand{\oalpha}{\smash{\overline{\alpha}}\vphantom{\alpha}}
\newcommand{\ogamma}{\smash{\overline{\gamma}}\vphantom{\gamma}}

\renewcommand{\R} {\mathbb{R}} 
\newcommand{\Z} {\mathbb{Z}} 
\renewcommand{\C} {\mathbb{C}} 
\newcommand{\N} {\mathbb{N}}

\newcommand{\Q} {\mathbb{Q}}

\newcommand{\seq}[3][{}]{\langle #2 \rangle_{#3}^{#1}}

\renewcommand{\mathcal}[1]{\mathscr{#1}}

\DeclareMathAlphabet{\mathbbb}{U}{bbold}{m}{n}
\providecommand*{\eu}%
{\ensuremath{\mathrm{e}}}
\providecommand*{\iu}%
{\ensuremath{\mathrm{i}}}

%% file: introduction.tex
\section{Introduction}

\paragraph*{The Skolem Problem}

An integer-valued \emph{linear recurrence sequence} \(\seq[\infty]{X_n}{n=0}\) satisfies a relation of the form
\begin{equation} \label{eq:rec}
	X_{n+d} = a_{d-1} X_{n+d-1} + \cdots + a_1 X_{n+1} + a_0 X_n
\end{equation}
for each \(n\in\N_0\).
Without loss of generality, we shall assume that each of the coefficients \(a_0, a_1,\ldots, a_{d-1}\in\Z\) and additionally that \(a_0\neq 0\).
We call \(d\) the \emph{length} of the recurrence relation and the \emph{order} of \(\seq{X_n}{n}\) is the length of the shortest relation satisfied by \(\seq{X_n}{n}\).
The polynomial \(f(x) = x^d - a_{d-1}x^{d-1} - \cdots - a_1 x - a_0\) is the \emph{characteristic polynomial} associated with relation \eqref{eq:rec}.
Given such a sequence, the \emph{Skolem Problem} \cite{everest2003recurrence,halava2005skolem} asks to determine whether there exists an \(n\in\N\) such \(X_n=0\).
The Skolem Problem is well-motivated with connections to research topics such as program verification \cite{ouaknine2015linear}.
Take, for example, the following linear loop \(P\) with inputs \(\vect{w},\vect{b} \in\Z^d\) and \(A\in\Z^{d\times d}\) where
\begin{equation} \label{eq:loop}
	P \colon \vect{v} \leftarrow \vect{w};\enspace \textbf{while}\enspace \vect{b}^\top \vect{v} \neq 0\enspace \textbf{do}\enspace \vect{v} \leftarrow A \vect{v}.
\end{equation}
Let \(\seq{X_n}{n}\) be the linear recurrence sequence with terms given by \(X_n = \vect{b}^\top A^n \vect{w}\).
It is clear that loop \(P\) terminates if and only if there exists an \(n\in\N_0\) such that \(X_n=0\).

\paragraph*{Motivation}
A recent resurgence of interest in the Skolem Problem (and related problems) has lead to the publication of a number of papers that consider restricted variants.
The resulting specialised decision procedures generally fall into two categories: those that consider an infinite subset of the natural numbers \cite{kenison2020skolem, luca2021universal} or those that restrict the class of linear recurrence sequences.
Our motivation is the latter type of specialisation and, in particular, a recent paper by Lipton et al.\ \cite{lipton2021skolem} that establishes the following theorem.
\newpage
\begin{theorem} \label{thm:main}
The Skolem Problem is decidable for the class of reversible integer linear recurrences of order at most seven.
\end{theorem}
An integer linear recurrence sequence \(\seq[\infty]{X_n}{n=0}\) is \emph{reversible} if it satisfies a recurrence relation of the form \eqref{eq:rec} such that \(a_0=\pm1\).
As observed in Lipton et al.~\cite{lipton2021skolem}, given an integer linear recurrence sequence \(\seq[\infty]{X_n}{n=0}\), the unique bi-infinite extension \(\seq[n=\infty]{X_n}{n=-\infty}\) has integral terms if and only if \(\seq[\infty]{X_n}{n=0}\) is reversible (this claim follows from a classical observation for Fatou rings \cite{Fat04}).
We can also characterise the subclass of while loops (as in \eqref{eq:loop}) naturally associated with reversible sequences:
the update matrix \(A\) with characteristic polynomial \(f\) is \emph{unimodular}; that is, \(A\) has integer entries and \(\det(A)=-f(0)=\pm 1\).
If \(A\) is unimodular, then \(A^{-1}\) also has integer entries.
Thus, again, \(\seq[n=\infty]{X_n}{n=-\infty}\) with each \(X_n = \vect{b}^\top A^n \vect{w}\) (as above) is integer-valued.

Unimodular matrices appear elsewhere in the dynamical systems literature.
Some classes lead to prototypical invertible maps with hyperbolic and ergodic properties; specifically, classes of \emph{linear toral automorphisms} \(T_A \colon \R^d \to \R^d\) given by \(T_A(\vect{x}) = A\vect{x}\)  \cite{katok1995introduction}.

A famous example of a reversible sequence is the Fibonacci sequence, which is defined by the initial values \(X_0=0\), \(X_1=1\) and for each \(n\in\N_0\), \(X_{n+2} = X_{n+1} + X_n\).
The Fibonacci sequence can be uniquely extended to a bi-infinite sequence of integer values \(\langle \ldots, 5, -3, 2, -1, 1, 0, 1, 1, 2, 3, 5, \ldots\rangle\).

In the sequel, we call the restricted variant of the Skolem Problem for reversible sequences the \emph{Reversible Skolem Problem}.

\paragraph*{Background}
Let us assess the current state of play with regards to the decidability of the Skolem Problem.
A classical result due to Skolem \cite{skolem1934verfahren} (which was later generalised by Mahler \cite{mahler1935taylor,mahler1956taylor}, and Lech \cite{lech1952recurring})
 states that \(\{n\in\N : X_n=0\}\) is the union of a finite set together with a finite number of (infinite) arithmetic progressions. 
The phenomenon that causes these vanishing arithmetic progressions is termed \emph{degeneracy}.
A sequence is \emph{degenerate} when one of the ratios of two distinct characteristic roots of the sequence is a root of unity.
These arithmetic progressions can be determined algorithmically and so, from the viewpoint of verification, to decide the Skolem Problem it suffices to consider \emph{non-degenerate} recurrence sequences---those sequences that have only finitely many zeros.
Indeed, this is where the difficulty lies: there is no known general method to compute this finite set. 
We refer the interested reader to Corollary 1.20 and Chapter 2 in \cite{everest2003recurrence} for further details.
In summary, all known proofs of the Skolem--Mahler--Lech Theorem (as it is now known)  are non-constructive and so the decidability of the Skolem Problem remains open.

Limited progress has been made on the decidability of the Skolem Problem when one considers linear recurrence sequences of low order.
Groundbreaking work by Mignotte, Shorey, and Tijdeman~\cite{mignotte1984distance}, and, independently, Vereshchagin~\cite{vereshchagin1985occurence} establish the following.
\begin{theorem} \label{prop:maxevalues}
The Skolem Problem is decidable for the class of non-degenerate linear recurrences with at most three simple characteristic roots that are maximal in modulus.
\end{theorem}
As a consequence, the Skolem Problem is decidable for linear recurrences of order at most four.
The aforementioned papers employ techniques from \(p\)-adic analysis and algebraic number theory and, in addition, Baker's theorem for
linear forms in logarithms of algebraic numbers.  
Unfortunately the route taken via Baker's Theorem does not appear to extend easily to recurrences of higher order. 

A class of recurrence sequences of order five that the state of the art cannot handle impedes further progress on the decidability of the Skolem Problem \cite{ouaknine2012decision}.
The minimal polynomial for each member of this class has four distinct roots \(\alpha, \oalpha, \beta, \obeta \in \C\) (two pairs of complex-conjugate roots) such that \(|\alpha|=|\beta|\), and a fifth real root \(\gamma\) of strictly smaller modulus.
Hence the terms of such a sequence \(\seq{X_n}{n}\) are given by an exponential polynomial of the form \(X_n = a(\alpha^n + \oalpha^n) + b(\beta +\obeta^n) + c\gamma^n\).
Here \(a,b,c\in\R\) are algebraic numbers and, as far as we are aware, there is  no known general procedure to determine \(\{n\in\N : X_n =0\}\) when  \(|a|\neq |b|\).

\paragraph*{Complexity}
In \cite{blondel2002skolem}, Blondel and Portier proved that the Skolem Problem is NP-hard.
As a brief aside, let us consider the complexity of the reversible variant of the Skolem Problem.
One of the questions considered by S.~Akshay et al.~\cite{akshay2017complexity} is the complexity of the Skolem Problem for the restricted class of linear recurrence sequences whose characteristic roots are all roots of unity (the so-called \emph{Cyclotomic Skolem Problem}).
Those authors showed, by a reduction from the Subset-Sum Problem, that the Cyclotomic Skolem Problem is NP-hard.
Because the characteristic polynomial associated with each element in this restricted class is given by a product of cyclotomic polynomials, it follows that each instance of the Cyclotomic Skolem Problem is an instance of the Reversible Skolem Problem.
Thus the Reversible Skolem Problem is also NP-hard.

\paragraph*{Contributions}
The main contribution in this note is an alternative proof of \autoref{thm:main}.
By comparison to the extensive case analysis employed in \cite{lipton2021skolem}, we use results in number theory for Galois conjugates that obey polynomial identities.
In particular, we make repeated use of a result due to Dubickas and Smyth \cite{dubickas2001remak} for algebraic integers that lie alongside all their Galois conjugates on two (but not one) concentric circles centred at the origin.

For context, Dubickas and Smyth's result is part of a large corpus of research on algebraic numbers whose conjugates lie on a conic or a union of conics.
Let \(\alpha\) be an algebraic number with Galois conjugates \(\alpha=\alpha_1,\ldots, \alpha_d\).
We call the set \(S(\alpha) := \{\alpha_1,\ldots, \alpha_d\}\) the \emph{conjugate set} of  \(\alpha\).
When \(S(\alpha)\) is a subset of the unit circle, a result of Kronecker's (a weaker version of \autoref{thm:kronecker}) proves that \(\alpha\) is a root of unity.
Number theorists have long-studied classes of algebraic integers where one or more of the conjugates leaves the circle; for example, a real algebraic integer \(\alpha\) is a \emph{Salem number} if \(\alpha>1\), \(\alpha^{-1}\in S(\alpha)\), and the remaining conjugates all lie on the unit circle, i.e., \(S(\alpha) = \{\alpha^{\pm 1}, \eu^{\pm\iu \theta_2}, \ldots, \eu^{\pm\iu \theta_d}\}\).

With regards to the Reversible Skolem Problem at order eight, we exhibit a family of recurrence sequences that, as far as we know, are not amenable to standard techniques and so decidability is very much open.
The authors of \cite{lipton2021skolem} also demonstrated a concrete family of examples in this regard.
It is interesting to note that the techniques used and families obtained are very different.
Each member of our family is an octic palindromic polynomial (the same is not true for the family of examples in \cite{lipton2021skolem}).
The coefficients of a \emph{palindromic} polynomial form a palindromic string of integers.
For example, two palindromes that are also family members are
\(x^8 + x^7 - x^6 + x^5 + 5x^4 + x^3 - x^2 + x + 1\) and
\(x^8 + x^7 - 3x^6 + x^5 + 9x^4 + x^3 - 3x^2 + x + 1\).
Using any modern computer algebra system, it is easy to verify that the roots of each polynomial satisfy the following.
First, the roots lie on two (but not one) concentric circles centred at the origin.
Second, no ratio of any two of its roots is a root of unity.
The calculations involved in preparing these (and later) examples were performed in PARI/GP  \cite{PARI2}.
In \autoref{ssec:galois}, we study the Galois groups of irreducible palindromic octics in the aforementioned family as a further investigation into the symmetries between their roots.

In Subsections~\ref{ssec:SRPP} and \ref{ssec:SPunitnorm}, we prove new decidability results on restricted variants of the Positivity and Skolem Problems using Dubickas and Smyth's theorem and make suggestions for further work in these directions.
For brevity, we refer to the Positivity Problem for the class of simple reversible integer linear recurrences as the \emph{Simple Reversible Positivity Problem}.
Here a linear recurrence is \emph{simple} if the associated characteristic polynomial has no repeated roots.
We have the following results:
	\begin{corollary} \label{cor:srpp}
	The Simple Reversible Positivity Problem is decidable for integer-valued linear recurrences of order at most ten.
	\end{corollary}
	\begin{corollary} \label{cor:spunitnorm}
	The Skolem Problem is decidable for rational-valued linear recurrences that satisfy a relation of the form		 \(X_{n+5} = a_{4} X_{n+4} + a_{3} X_{n+3} + a_{2} X_{n+2} + a_1 X_{n+1} \pm X_n\)
 with \(a_1, a_2, a_3, a_4 \in\Q\).
	\end{corollary}

\paragraph*{Structure}
The remainder of this paper is structured as follows.
In the next section we review necessary preliminary material.
In \autoref{sec:decidability}, we give a new and novel proof of \autoref{thm:main}.
In \autoref{sec:octics}, we construct a family of octic palindromes that shows the current state of the art cannot settle decidability of the Reversible Skolem Problem at order eight and then discuss the Galois groups associated with the irreducible members of this family.
In the final section, \autoref{sec:future}, we discuss directions and motivate this discussion with the proofs of Corollaries~\ref{cor:srpp} and \ref{cor:spunitnorm}.

%% file: preliminaries.tex
\section{Preliminaries}

\subsection{Recurrence Sequences}

A sequence \(\seq[\infty]{X_n}{n=0}\) of integers satisfying a recurrence relation of the form \eqref{eq:rec} with fixed integer constants \(a_0,a_1,\ldots, a_{d-1}\) such that \(a_0\neq 0\) is a \emph{linear recurrence sequence}.
The sequence \(\seq{X_n}{n}\) is then wholly determined by the recurrence relation and the initial values \(X_0, X_1,\ldots, X_{d-1}\).
The polynomial \(f(x) = x^d - a_{d-1}x^{d-1} - \cdots - a_1 x - a_0\) is the \emph{characteristic polynomial} associated with relation \eqref{eq:rec}.
From our earlier definition, it is clear that \(\seq{X_n}{n}\) is reversible if and only if \(f(0) = \pm 1\). %
There is a recurrence relation of minimal length associated to \(\seq{X_n}{n}\) and we call the characteristic polynomial of this minimal length relation the \emph{minimal polynomial} of \(\seq{X_n}{n}\). 
The \emph{order} of a linear recurrence sequence is the degree of its minimal polynomial.

Let \(f\) be the minimal polynomial of a linear recurrence sequence \(\seq{X_n}{n}\) and \(K\) the splitting field of \(f\).
The polynomial \(f\) factorises as a product of powers of distinct linear factors like so \(f(x) = \prod_{\ell=1}^m (x-\lambda_\ell)^{n_\ell}
\).
The constants \(\lambda_1,\lambda_2,\ldots, \lambda_m\in K\) are the \emph{characteristic roots} of \(\seq{X_n}{n}\) with multiplicities \(n_1,n_2,\ldots, n_m\).
One can realise the terms of a linear recurrence sequence as an \emph{exponential polynomial} \(X_n = \sum_{\ell =1}^m p_\ell(n) \lambda_\ell^n\) where the \(\lambda_\ell\) are the aforementioned characteristic roots of \(\seq{X_n}{n}\) and the polynomial coefficients \(p_\ell\in K[x]\) are determined by the initial values.
We say a characteristic root of \(\seq{X_n}{n}\) is \emph{dominant} if in the set of characteristic roots of \(\seq{X_n}{n}\) it is maximal in modulus.
Thus, by \autoref{prop:maxevalues}, the Skolem Problem is decidable for the class of non-degenerate linear recurrence sequences with at most three dominant characteristic roots \cite{mignotte1984distance, vereshchagin1985occurence}.
\subsection{Number Theory}

We shall assume some familiarity with Galois theory and the theory of number fields.
The necessary background material can be found in a number of standard textbooks~\cite{cohen1993computational, stewart2016algebraic}.

Recall the following theorem due to Kronecker \cite{kronecker1857}.
			\begin{theorem} \label{thm:kronecker}
				Let \(f\in\Z[x]\) be a monic polynomial such that \(f(0)\neq 0\).  
				Suppose that all the roots of \(f\) have absolute value at most \(1\), then \(f\) is a product of cyclotomic polynomials.  
				Therefore all the roots of \(f\) are roots of unity.
			\end{theorem}
		
Thus, if \(f\in\Z[x]\) is the characteristic polynomial of a reversible linear recurrence sequence such that the roots of \(f\) all lie in the unit disk \(\{z\in\C : |z|\le 1\}\).
Then the roots of \(f\) are all roots of unity.
It follows that the associated recurrence sequence is either order one (and is thus constant) or degenerate.
In either case the Skolem Problem is decidable.
Thus in the sequel we shall always assume, without loss of generality, that the dominant roots of \(f\) lie on a circle  with radius strictly larger than \(1\).

In the sequel, our construction of an infinite family of octics uses the following corollary of Vieta's formulae.
	\begin{lemma} \label{lem:algintegers} Suppose that \(f\in\Z[x]\) is a monic irreducible polynomial such that \(f(x) = \prod_{i=1}^d (x-\lambda_i)\).
	Let \(f_n(x) := \prod_{i=1}^d (x-\lambda_i^n)\).
	Then \(f_n\in\Z[x]\) for each \(n\in\N\).	
	\end{lemma}
	\begin{proof}
		The coefficients of the polynomial \(f_n\) are determined by symmetric polynomials in \(d\) variables.
		By the fundamental theorem of symmetric polynomials, each symmetric polynomial is given by a \(\Z\)-linear combination of elementary symmetric polynomials.
		The result follows as a straightforward application of Vieta's formulae and the evaluation of elementary symmetric polynomials over conjugate algebraic integers.
	\end{proof}

The roots of an irreducible polynomial are necessarily Galois conjugates.
We use the term \emph{conjugate ratios} for the ratios between two distinct roots of an irreducible polynomial.
A non-zero algebraic number \(\alpha\) is \emph{reciprocal} if \(\alpha\) is conjugate to \(\alpha^{-1}\).  Let \(\tau\) be a Salem number whose minimal polynomial has degree \(2d\) or a reciprocal quadratic.  
In the former case, \(S(\tau) = \{\tau^{\pm 1}, \tau_2^{\pm 1},\ldots, \tau_d^{\pm 1}\}\) and in the latter, \(S(\tau) = \{\tau^{\pm 1}\}\).
An algebraic number \(\psi\) is a \emph{Salem half-norm} if \(\psi = \tau^{\epsilon_1} \tau_2^{\epsilon_2} \cdots \tau_d^{\epsilon_d}\) for some such \(\tau\) and \(\varepsilon_j = \pm 1\) for each \(j\in\{1,\ldots, d\}\).
The properties of Salem half-norms are discussed further in \cite{dubickas2001remak}.

In the ring of algebraic integers of a given number field \(K\), \(\gamma\in K\) is a \emph{unit} (sometimes an \emph{algebraic unit}) if it has a multiplicative inverse \(\delta\) so that \(\gamma \delta = \delta \gamma =1\).
Let \(\alpha\in K\) be an algebraic integer.
Then the constant coefficient of the minimal polynomial \(f\in\Z[x]\) of \(\alpha\) is equal to \(\pm 1\) if and only if \(\alpha\) is a unit. 
This observation follows easily from norm considerations and the fact that the constant coefficient of \(f\) (up to sign) is given by the product of \(\alpha\) and its Galois conjugates.
Since the characteristic polynomial \(f\in\Z[x]\) of a reversible sequence \(\seq{X_n}{n}\) has constant coefficient \(\pm 1\), we deduce that each of the characteristic roots of \(\seq{X_n}{n}\) is a unit.
In the sequel we make frequent use of the following simple observation.
	\begin{lemma} \label{lem:inunitdisk}
		If \(\alpha\) is an algebraic unit that lies on the circle \(|z|=R\) with \(R>1\), then a conjugate of \(\alpha\) lies in the interior of the unit disk.
	\end{lemma}
Key to the proofs in the sequel is a powerful result due to Dubickas and Smyth \cite{dubickas2001remak} concerning polynomial identities between the roots of irreducible polynomials.	
\autoref{thm:bicycle} gives necessary conditions for a unit in the algebraic integers and all its Galois conjugates to lie on two (but not one) concentric circles centred at the origin	\cite{dubickas2001remak}.
In fact, Dubickas and Smyth prove a far more general result \cite[Theorem 2.1]{dubickas2001remak}, but we need only the specialised version for units below.
	\begin{theorem}	\label{thm:bicycle}
	 Suppose that \(\alpha\) is a unit in the algebraic integers of degree \(d\) lying, with all its Galois conjugates, on two circles \(|z|=r\) and \(|z|=R\), but not just one.
	 Without loss of generality assume that at most half of the conjugates lie on \(|z|=r\).
	 Then, one of the following holds:
	 	\begin{enumerate}
	 		\item \(d=3m\), \(R= r^{-1/2}\) such that there are \(d/3\) conjugates of \(\alpha\) on \(|z|=r\), and the remaining \(2d/3\) lie on \(|z|=r^{-1/2}\).
	 		Assume, without loss of generality, \(|\alpha|=r\).
	 		Then, in addition, there exists an \(n\in\N\) such that \(\alpha^n\) is a real, but non-totally real, cubic unit.
	 		\item \(d=2m\), \(R = r^{-1}\) where \(R>1\) without loss of generality, and \(d/2\) Galois conjugates of \(\alpha\) lie on each circle.
	 		Further, there exists an \(n\in\N\) such that \(\alpha^n =: \psi\)  is a Salem half-norm defined by a Salem number or a reciprocal quadratic.
	 	\end{enumerate}
	\end{theorem}

Let us explain the term \emph{totally real} in the last theorem.
An algebraic number \(\alpha\) is \emph{totally real} if \(\alpha\) and all its Galois conjugates are real.
A number field 
\(K\) is \emph{totally real} if \(K = \Q[\alpha]\) such that \(\alpha\) is totally real.

From this point to the end of the subsection,  the terminology and results we recall are used only in \autoref{sec:octics} (and so are not required for the proof of \autoref{thm:main}).

A field is \emph{Kroneckerian} if it is either 
    a totally real algebraic number field or a totally imaginary quadratic extension of a totally real field.
In the sequel, we make use of the following observation about complex conjugation lying in the centre of the Galois group of a Kroneckerian field (see \cite[Chapter 6]{schinzel2000polynomials}).    
    \begin{corollary} \label{cor:schinzel}
     A number field \(K\) is Kroneckerian if and only if for every \(\alpha\in K\) one has \(\overline{\alpha}\in K\) and for every embedding \(\sigma\) of \(K\) into \(\C\) one has \(\overline{\alpha^\sigma} = \overline{\alpha}^\sigma\).
    \end{corollary}

 A unit is \emph{unimodular} if it lies on the unit circle in \(\C\).
In \cite{MacCluer1975}, MacCluer and Parry prove that a normal imaginary field contains unimodular units other than roots of unity exactly when its real subfield is not normal over \(\Q\).
 Daileda \cite{daileda2006} generalises this result and provides the following classification of the number fields that have unimodular units that are not roots of unity.
 Recall that a number field \(K\) is a \emph{CM-field} if \(K\) is a totally complex quadratic extension of a totally real field.
	\begin{theorem} \label{thm:daileda}
		Let \(K\) be a number field closed under complex conjugation. 
		Then \(K\) contains unimodular units that are not roots of unity if and only if \(K\) is imaginary and not a CM-field.
	\end{theorem}

\subsection{Group Theory}

In the sequel we employ the notation \(S_n\) for the symmetric group on \(n\) elements, \(A_n\) for the alternating group on \(n\) elements, \(D_n\) for the Dihedral group of order \(2n\), \(C_n\) for the cyclic group on \(n\) elements, and \(K_4\) for the Klein \(4\)-group.

The action of \(G\) on a set \(X\) is \emph{transitive} if for every pair \(x,y\in X\) there is a \(g\in G\) such that \(gx = y\); that is to say, there is a single group orbit.
We note here (and again later) that \(S_4, A_4, D_4, C_4\), and \(K_4\) are the transitive subgroups of \(S_4\).

%% file: decidability.tex
\section{Proof of \autoref{thm:main}} \label{sec:decidability}

We briefly outline our route to proving \autoref{thm:main}.
We claim that if \(\seq{X_n}{n}\) is a non-degenerate reversible integer recurrence sequence of order at most seven,
then \(\seq{X_n}{n}\) has at most three dominant characteristic roots.
The decidability of the Skolem Problem for such instances then follows from \autoref{prop:maxevalues}.
The above claim follows as a corollary of the next theorem.

\begin{theorem} \label{thm:threeroots}
 No monic polynomial \(f\in\Z[x]\) with constant coefficient \(\pm 1\) of degree at most seven satisfies the following two properties:
	\begin{enumerate}[itemindent=1.2em] 
	\renewcommand{\labelenumi}{\textcolor{lipicsGray}{ \sffamily\bfseries\upshape\mathversion{bold}(\theenumi)}}
    \renewcommand{\theenumi}{H\arabic{enumi}}
		\item { \label{h1} \(f\) has at least four distinct dominant roots; and} 
		\item { \label{h2} no quotient of two distinct roots of \(f\) is a root of unity.}
	\end{enumerate}
\end{theorem}

Thus all that remains is to prove \autoref{thm:threeroots}.
This result is an immediate consequence of the sequence of Propositions~\ref{prop:unit5}, \ref{prop:unit6}, and \ref{prop:unit7} below.
In each of the proofs of these propositions we play a similar game: we assume, for a contradiction, that there exists a polynomial \(f\in\Z[x]\) (of degree five, six, or seven respectively) that satisfies hypotheses \ref{h1} and \ref{h2}.
We show that such a candidate is necessarily irreducible.
We then employ \autoref{thm:bicycle} to derive a contradiction: such a candidate cannot satisfy both \ref{h1} and \ref{h2} and, at the same time, satisfy the restrictive root identities prescribed by \autoref{thm:bicycle}.

For the avoidance of doubt, this route to \autoref{thm:main} is similar to that carved out by \cite{lipton2021skolem}.
The contribution of this paper is the novel application of \autoref{thm:bicycle}.
Indeed, our assumption, that each of the characteristic roots of a recurrence in our class of non-degenerate reversible sequences is a unit, leads to (rather startling) restrictive polynomial relations between Galois conjugates.

We begin our sequence of propositions.
\begin{proposition} \label{prop:unit5}
	No monic polynomial \(f\in\Z[x]\) of degree at most five with constant term \(\pm 1\) satisfies hypotheses \ref{h1} and \ref{h2}.
\end{proposition}
\begin{proof}
 Assume, for a contradiction, that such an \(f\) of degree \(d\le 5\) exists.
 By \autoref{thm:kronecker}, the dominant roots lie on the circle \(|z|=R\) for some \(R>1\) for otherwise the roots of \(f\) are necessarily all roots of unity, which is not permitted under hypothesis \ref{h2}.
 Each of the dominant roots of \(f\) is a unit in the algebraic integers and so, by \autoref{lem:inunitdisk}, has a Galois conjugate in the interior of the unit disk.
 In order that \(f\) satisfies hypothesis \ref{h1}, we conclude that \(f\) has four simple dominant roots and a single non-dominant root.
 Since \(f\) has degree \(d=5\) and \(f(0)=\pm 1\), all the dominant roots of \(f\) are (Galois) conjugate to the single non-dominant root, \(f\) is irreducible.

Let \(\alpha\) be a root of \(f\).
Then \(\alpha\) is an algebraic integer of degree \(5\), a unit, and lies with all its Galois conjugates on two circles.
We apply \autoref{thm:bicycle} to the quintic \(f\) and find that  \(5\) is either even, or a multiple of \(3\), a contradiction. 
\end{proof}

\begin{remark}
The application of \autoref{thm:bicycle} in the proof of \autoref{prop:unit5} is excessive (even if the derived contradiction is rather satisfying).
By comparison, the approach in Lipton et al.~\cite{lipton2021skolem} is direct.
We reproduce the final part of those authors' proof below as it gives a gentle introduction to some of the techniques we apply in \autoref{prop:unit6} and \autoref{prop:unit7}.

\begin{proof}[Proof of \autoref{prop:unit5} (cf.~\cite{lipton2021skolem})]
Let \(\alpha, \oalpha, \beta, \obeta\) be the four dominant roots of \(f\) and \(\rho\) the non-dominant root of \(f\).
Since \(f\) is irreducible, the Galois group \(G\) of \(f\) acts transitively on the roots of \(f\).
Thus there exists a \(\sigma\in G\) such that \(\sigma(\alpha)=\rho\).
The element \(\sigma\) must preserve the equality \(\alpha\oalpha = \beta\obeta\) and so we have \(\rho \sigma(\oalpha) = \sigma(\beta)\sigma(\obeta)\).
We derive a contradiction: \(|\rho| |\sigma(\oalpha)| \neq |\sigma(\beta)| |\sigma(\obeta)|\) since the two roots \(\sigma(\beta)\) and \(\sigma(\obeta)\) on the right-hand side are necessarily dominant.
\end{proof}
\end{remark}

\begin{proposition} \label{prop:unit6}
	No monic polynomial \(f\in\Z[x]\) of degree six with constant term \(\pm 1\) satisfies hypotheses \ref{h1} and \ref{h2}.
\end{proposition}
\begin{proof}
	Assume, for a contradiction, that such an \(f\) exists.
	As in the proof of \autoref{prop:unit5},
	\(f\) has at least four simple dominant roots \(\alpha, \oalpha, \beta, \obeta\) that lie on a circle \(\{z\in\C : |z|=R\}\) for some \(R>1\) as complex-conjugate pairs.
	We note that the group \(G\) of automorphisms of the splitting field of \(f\) must preserve the equality \(\alpha\oalpha = \beta\obeta\).
	Because \(\alpha\) is a unit on \(|z|=R\), by \autoref{lem:inunitdisk} there is both a root \(\gamma\) of \(f\) that lies in the unit disk and a permutation \(\sigma\in G\) such that \(\sigma(\alpha) = \gamma\).
	Now consider  \(\sigma(\alpha)\sigma(\oalpha) = \sigma(\beta)\sigma(\obeta)\).
	It is straightforward to elicit a contradiction that breaks this equality if either  \(\sigma(\oalpha)\) is non-dominant, or both \(\sigma(\beta)\) and \(\sigma(\obeta)\) are dominant.
	Thus we can assume that \(f\) has two non-dominant roots and further that they are of equal modulus.
	Clearly these two roots \(\gamma, \ogamma\) are a complex-conjugate pair by \ref{h2}.
	Combining these observations of the roots of \(f\), we quickly deduce that \(f\) is necessarily irreducible.

	Because the roots of the irreducible polynomial \(f\in\Z[x]\) lie on two concentric circles, we can apply \autoref{thm:bicycle}.
	Thus there is a non-dominant root, \(\gamma\) say, and \(m\in\N\) such that \(\gamma^m\) is a real cubic unit.
	It follows that \(\gamma^m / \ogamma^m =1\) and so one of the conjugate ratios of \(f\) is a root of unity, which contradicts hypothesis \ref{h2}.
\end{proof}

\begin{proposition} \label{prop:unit7}
	No monic polynomial \(f\in\Z[x]\) of degree seven with constant term \(\pm 1\) satisfies hypotheses \ref{h1} and \ref{h2}.
\end{proposition}
\begin{proof}
	Assume, for a contradiction, that such an \(f\) exists.
	As in the proof of \autoref{prop:unit5},
	\(f\) has at least four simple dominant roots \(\alpha, \oalpha, \beta, \obeta\) that lie on a circle \(\{z\in\C : |z|=R\}\) for some \(R>1\) as complex-conjugate pairs.
	Mutatis mutandis, one can use the methods in the proof of \autoref{prop:unit6} to make the following two deductions.
	First, \(f\) is irreducible.
	Second, \(f\) has precisely one real root \(\delta\) and another complex-conjugate pair of roots \(\gamma, \ogamma\).
	Additionally, we have that \(\delta, \gamma, \ogamma\) are all non-dominant roots of \(f\) such that \(|\gamma|\neq |\delta|\).
	We note that if one supposes, for a contradiction, that the septic polynomial \(f\) has  \(|\gamma|= |\delta|\) or \(|\alpha|=|\delta|\), then, by \autoref{thm:bicycle}, it follows that \(7\) is either even, or a multiple of \(3\).

Because \(f\) is irreducible, the Galois group of \(f\) acts transitively on the roots of \(f\).
Thus there is a permutation \(\sigma\) in the Galois group of \(f\) such that \(\sigma(\alpha)=\delta\).
We know that \(\sigma\) must preserve the equality \(\alpha\oalpha = \beta\obeta\) and so we have \(|\delta\sigma(\oalpha)|=|\sigma(\beta)\sigma(\obeta)|\).
The left-hand side is equal to one of \(|\delta||\gamma|\) or \(|\delta|R\) depending on \(\sigma(\oalpha)\).
There are three cases to consider for the right-hand side.
The roots \(\sigma(\beta)\) and \(\sigma(\obeta)\) are either both dominant, both non-dominant, or one of each.
It is clear that the modulus in each of these (respective) cases \(R^2\), \(|\gamma|^2\), and \(|\gamma|R\) breaks the aforementioned equality between the left- and right-hand sides, a contradiction.
\end{proof}
Hence we have proved \autoref{thm:threeroots}, as required.

%% file: octics.tex
\section{Palindromic Octics} \label{sec:octics}

A monic polynomial \(f\in\Z[x]\) is \emph{palindromic} if its coefficients string together to form a palindrome.  That is to say, 	for \(f(x) = x^d + a_{d-1} x^{d-1} + \cdots + a_1 x + 1\) we have that \(a_k = a_{d-k}\) for each \(k\in\{0,1,\ldots, d\}\).
In other sources, the term \emph{self-reciprocal} is sometimes used since if \(\alpha\) is a root of \(f\) then \(\alpha^{-1}\) is also a root of \(f\).
If \(\alpha\in\C\) is a root of \(f\), then so are \(\bar{\alpha}, 1/\alpha\), and \(1/\bar{\alpha}\).
Further, when \(\alpha\) is neither real nor lies on the unit circle then these four roots are distinct.
The class of palindromic polynomials appear in areas across mathematics and computer science in fields such as: coding theory, algebraic curves over finite fields, knot theory, and linear feedback shift registers, to name but a few (see the survey \cite{joyner2013zeros}).

\begin{remark}
Let us further motivate the study of the \emph{Palindromic Skolem Problem} (the Skolem Problem restricted to the study of sequences with palindromic characteristic polynomials).
From the viewpoint of dynamical systems, we observe that the equations of motion \eqref{eq:rec} governing a recurrence sequence with palindromic characteristic polynomial (hereafter a \emph{palindromic recurrence sequence}) possess a \emph{time-reversing symmetry}; that is to say, the recurrence relation is invariant under the time reversal map \(n \mapsto -n\).
More concretely, let \(\seq[\infty]{X_n}{n=0}\) and \(\seq[\infty]{Y_{-n}}{n=0}\) be recurrence sequences satisfying the palindromic relations
\begin{align*}
	X_{n+d} &= a_{d-1} X_{n+d-1} + \cdots + a_{d-1} X_{n+1} + X_n \qquad \text{and} \\
	Y_{-n-d} &= a_{d-1} Y_{-n-d+1} + \cdots + a_{d-1} Y_{-n-1} + Y_{-n},
\end{align*}
respectively.
Then \(Y_{-n} = X_n\) for each \(n\in\N_0\) if \(Y_{-m} = X_m\) for each \(m\in\{0,1,\ldots, d-1\}\).

For further information on the topic of time-reversing symmetries, we refer the interested reader to the exposition in Lamb's article \cite{lamb1992reversing} and the recent survey by Baake \cite{baake2018reversing} (and the references therein).
\end{remark}

\subsection{Hard instances of the Reversible Skolem Problem at order eight}

In \autoref{cor:infinitefamily}, we construct an infinite family of palindromic octics that satisfy both \ref{h1} and \ref{h2}.
This result blocks any obvious attempt to settle decidability of the Reversible Skolem Problem at order eight with current state-of-the-art techniques.
Another family of octics that satisfy both \ref{h1} and \ref{h2} is discussed in Lipton et al.~\cite{lipton2021skolem}.
The additional restriction (palindromic coefficients) we require demonstrates a refinement in the discussion of hard instances of the Reversible Skolem Problem at order eight.

	\begin{proposition} \label{cor:infinitefamily}
	 There are infinitely many palindromic octics in \(\Z[x]\) that satisfy both \ref{h1} and \ref{h2}.
	\end{proposition}
	\begin{proof}	
	Let \(f\in\Z[x]\) be such a palindromic octic with four simple dominant roots (that is, two complex-conjugate pairs) that lie on a circle of radius \(r>1\).
	Write \(f(x) = \prod_{i=1}^8 (x-\lambda_i)\) and for each \(n\in\N\), define \(f_n(x) := \prod_{i=1}^8 (x-\lambda_i^n)\).
	By the symmetries in the roots \(\lambda_i^n\), it is straightforward to verify the following three observations.
	First, each octic \(f_n\) in the sequence \(\seq[\infty]{f_n}{n=1}\) is palindromic.
	Second,
	each octic \(f_n\) satisfies both \ref{h1} and \ref{h2}.
	Third, by \autoref{lem:algintegers}, \(f_n\in\Z[x]\) for each \(n\in\N\).

	We finish the proof by showing that  there are infinitely many distinct polynomials in the sequence \(\seq[\infty]{f_n}{n=1}\).
	Let \(r\eu^{\pm\iu \theta}, r\eu^{\pm\iu \psi}\) be the four dominant roots of \(f\).
	Let \(f_n(x) = x^8 + a_{7,n} x^7 + \cdots + 1\).
	Then, by Vieta's formulae, \(a_{7,n}\) is given by the sum of the \(n\)th powers of the roots of \(f\); that is,
		\begin{equation*}
			a_{7,n} = (r^n + r^{-n})(\eu^{n\iu\theta} + \eu^{-n\iu\theta} + \eu^{n\iu\psi} + \eu^{-n\iu\psi}).
		\end{equation*} 
	Assume, for a contradiction, that the non-degenerate integer linear recurrence sequence \(\seq{a_{7,n}}{n}\) takes only finitely many values.
	By the Pigeonhole Principle, there is an \(a\in\Z\) and an infinite subsequence \(\seq{n_k}{k}\) of natural numbers such that  \(a_{7,n_k}= a\) for each \(k\in\N\).
	We make two observations.
	First, the integer linear recurrence sequence \(\seq{a_{7,n} - a}{n}\) has characteristic polynomial \(f(x)(x-1)\) (see \cite[Theorem 1.1]{everest2003recurrence}) and so is non-degenerate.
	Second, the sequence \(\seq{a_{7,n} - a}{n}\) vanishes infinitely often.
	We have a contradiction: a non-degenerate linear recurrence sequence has only finitely many zero terms.
	It follows that there are infinitely many distinct palindromic octics in the sequence \(\seq{f_{n}}{n}\), as desired.
	\end{proof}

\begin{remark}
A \emph{symplectic matrix} \(M\) is a \(2\ell \times 2\ell\) matrix that preserves the \emph{symplectic form}  \(J\) in \(2\ell\) dimensions; that is to say, \(M^\top J M = J\) where
	\begin{equation*}
		J = \begin{pmatrix}
				0 & I_\ell \\ -I_\ell & 0
				\end{pmatrix}.
	\end{equation*}
It is well-known that the characteristic polynomial of a symplectic matrix is a palindromic polynomial. 
In fact, the following constructive result proves the converse.
Rivin \cite[Theorem A.1]{rivin2008walks} attributes the result to Kirby \cite{kirby1969integer}\footnote{The author was unable to access Kirby's article in order to verify this attribution.}.
	\begin{theorem} \label{thm:symplectic}
	For each monic palindromic polynomial \(f\in \Z[x]\) of degree \(2\ell\), there is a symplectic matrix \(M\in \Z^{2\ell\times 2\ell}\) such that \(\det(x I_{2\ell} - M) = f(x)\).
	\end{theorem}
	Note the group of \(2\ell \times 2\ell\) symplectic matrices with entries in \(\Z\) is closed under multiplication.
	This observation leads to an alternative proof of \autoref{cor:infinitefamily} as follows.
	Let \(f\in\Z[x]\) be an octic palindrome satisfying \ref{h1} and \ref{h2}.
	Now let  \(M\in \Z^{8\times 8}\) be a symplectic matrix with characteristic polynomial \(f\).
	Consider the sequence \(\seq{f_n}{n}\) where \(f_n(x) := \det(xI_{2\ell} - M^n)\in\Z[x]\).
	As before, we need to verify that each polynomial in this sequence satisfies the root assumptions  \ref{h1} and \ref{h2}.
	We then proceed in a similar fashion to the proof of \autoref{cor:infinitefamily}, in order to generate an infinite sequence of octic palindromes with the desired properties.
\end{remark}

\subsection{Galois theory of octic palindromes} \label{ssec:galois}

It is interesting to consider the root symmetries of the characteristic polynomials of reversible linear recurrence sequences.
In this subsection we focus our attention on the Galois groups of octic palindromes; in particular, those underlying hard instances of the Reversible Skolem Problem.
We present a new result (\autoref{thm:keilthy}).
The problem of root symmetries follows naturally from our approach to \autoref{thm:main} where we explored polynomial identities between roots of certain irreducible polynomials.
Such results are also motivated by hard open problems such as (polynomial) invariant generation and loop synthesis in the field of program verification.
The task of computing the Galois groups of irreducible polynomials in \(\Z[x]\) is well-known.
Families of polynomials associated with reversible sequences include the cyclotomic and Salem polynomials.
The authors of \cite{christopoulos2010galois} discuss ramifications to the Galois group of moving two of the Galois conjugates off of the unit circle.
For the interested reader, two accounts of the Galois theory of palindromic polynomials are \cite{viana2002galois, rivin2015large}. 
\lq{Generically\rq,} the Galois group of a palindromic polynomial of degree \(2d\) is \(S_d \wr C_2\) (the \emph{signed permutation group} or 
the \emph{hyperoctahedral group}).
In the case \(d=4\), the order of \(S_4 \wr C_2\) is 384.

Here we shall consider the Galois groups associated with irreducible octic palindromes of the form constructed in \autoref{cor:infinitefamily}.
Recall that each polynomial in that family has roots of the form \(r^{\pm 1}\eu^{\pm\iu \theta}, r^{\pm 1}\eu^{\pm\iu \psi}\).
When a polynomial with this root distribution is irreducible, some of the powers of these roots are given by Salem half-norms (\autoref{thm:bicycle}).
\begin{theorem} \label{thm:keilthy}
	Suppose that \(f\in\Z[x]\) is an irreducible and simple palindromic octic with four dominant roots such that none of its conjugate ratios are roots of unity (so that \ref{h1} and \ref{h2} are satisfied).
	Then, the associated Galois group is isomorphic to either \(S_4\times C_2\) or \(A_4\times C_2\).
\end{theorem}
\begin{proof}
Let \(\alpha, \oalpha, \beta, \obeta\) be the four dominant roots of \(f\).
As before, we can assume without loss of generality that these roots lie on a circle \(|z|=R\) with \(R>1\).
The Galois group of the splitting field of \(f\) is necessarily a subgroup of the group of automorphisms \(G\) of the set \(\{\alpha^{\pm 1}, \oalpha^{\pm 1}, \beta^{\pm 1}, \obeta^{\pm 1}\}\subset \C\).
Further, the Galois group necessarily contains permutations that fix the relations \(x x^{-1} =1\) for \(x\) in the set and \(\alpha  \alpha^{-1} - \beta \beta^{-1}= 0\).
We note that \(G\) also contains the \emph{inversion map} \(x\mapsto x^{-1}\).
The action of \(G\) on the roots induces an action of (a subgroup of) \(S_4\) on the set of \(C_2\) orbits \(\{x, x^{-1}\}\), 
which we lift to an action of (a subgroup of) \(S_4\) on the roots.

It is easily verified that the transposition \((12)\) induces one of the maps
	\begin{equation*}
		(\alpha, \oalpha, \beta, \obeta) \mapsto  (\oalpha, \alpha, \beta, \obeta), \quad \text{or} \quad
		(\alpha, \oalpha, \beta, \obeta)  \mapsto (\oalpha^{-1}, \alpha^{-1}, \beta^{-1}, \obeta^{-1}).
	\end{equation*}
Both of these maps commute with inversion and we note the latter is obtained from the former by applying the inversion map.
Thus, without loss of generality, we take \((12)\) as the former map.

In fact, every transposition can be lifted and commutes with inversion.
Hence \(G= S_4 \times C_2\).
The transformations are as follows:
	\begin{align*}
		(1 2) \colon& (\alpha, \oalpha, \beta, \obeta) \mapsto (\oalpha, \alpha, \beta, \obeta) \\
		(1 3) \colon& (\alpha, \oalpha, \beta, \obeta)  \mapsto (\beta^{-1}, \oalpha, \alpha^{-1}, \obeta) \\
		(1 4) \colon& (\alpha, \oalpha, \beta, \obeta)  \mapsto (\obeta^{-1}, \oalpha, \beta, \alpha^{-1}) \\
		(2 3) \colon& (\alpha, \oalpha, \beta, \obeta)  \mapsto (\alpha, \beta^{-1}, \oalpha^{-1}, \obeta) \\
		(2 4) \colon& (\alpha, \oalpha, \beta, \obeta)  \mapsto (\alpha, \obeta^{-1}, \beta, \oalpha^{-1}) \\
		(3 4) \colon& (\alpha, \oalpha, \beta, \obeta)  \mapsto (\alpha, \oalpha, \obeta, \beta).
	\end{align*}
By way of explanation, these transformations are deduced as follows.
First,  when choosing a lift of a transposition from the set of \(C_2\)-orbits to the set of roots, one can invert an even number of the orbit pairs \(\{x, x^{-1}\}\).
Second, the choice of inverting none or all four orbit pairs differs only by the inversion map.
Similarly, the two options when inverting two of the orbit pairs differ only by inversion.
A transposition \(\sigma\) necessarily preserves \(\alpha \oalpha - \beta \obeta = 0\).
Some careful accounting shows that permutations such as  \(\sigma(\oalpha)=\beta\) (or \(\sigma(\oalpha)=\obeta\)) lead to \(\alpha \beta - \overline{\alpha\beta}=0\) (or \(\alpha\obeta - \oalpha\beta=0\)) and so \(\alpha\beta \in\R\) (or \(\alpha\obeta\in\R\)).
Such conclusions contradict our assumptions on the conjugate ratios of \(f\).
Similar arguments lead to the conclusion that there is a unique (that is to say, unique up to inversion) automorphism of the roots that preserves the aforementioned relations.
Hence the Galois group is a subgroup of \(S_4 \times C_2\).

We make the following useful observations.
Firstly, the Galois group is a transitive subgroup of \(S_4\times C_2\) because \(f\) is irreducible.
Secondly, since the action of \(C_2\) on the roots is \emph{free} (given \(g,h\in C_2\) and a root \(x\) with \(gx = hx\) then necessarily \(g=h\)), the transitive subgroups of \(S_4\times C_2\) are precisely the direct product of \(C_2\) and a transitive subgroup of \(S_4\).
Finally, we note that \((12)(34)\), representing complex conjugation, is certainly an element of the Galois group.

The transitive subgroups of \(S_4\) are isomorphic to \(S_4, A_4, D_4, C_4\), and \(K_4\).
Because \(C_4\) and \(K_4\) are Abelian groups, \((12)(34)\) either lies in the centre of each group or the subgroup does not contain \((12)(34)\).
In the former case we deduce that the splitting field is a Kroneckerian field (by \autoref{cor:schinzel}) and so the splitting field is either CM or totally real.
The field cannot be totally real by our assumption that \(f\) has non-real roots.
The field is also not a CM-field for otherwise the conjugate ratios are necessarily roots of unity by \autoref{thm:daileda}.

We now focus on eliminating the possibility that the Galois group of \(f\) is \(D_4\times C_2\).
There are three conjugate subgroups of \(S_4\) that are isomorphic to \(D_4\):
\begin{align*}
	D_{4,0} &= \{e, (12), (34), (12)(34), (13)(24), (14)(23), (1324), (1423)\} = \langle (1324),(12) \rangle, \\
	D_{4,1} &= \{e, (13), (24), (13)(24), (12)(34), (14)(23), (1234), (1432)\} =\langle (1234),(13)\rangle,\ \text{and} \\
	D_{4,2} &= \{e, (14), (23), (14)(23), (12)(34), (13)(24), (1243), (1342)\} = \langle (1243),(14) \rangle.
\end{align*}
Note \((12)(34)\) lies in the centre of \(D_{4,0}\) and so we can once again employ \autoref{thm:daileda} to deduce that \(D_{4,0} \times C_2\) cannot be the Galois group of the splitting field of \(f\).
Assume, for a contradiction, that \(D_{4,1}\times C_2\) is the Galois group of the splitting field of \(f\).
It is easily verified that \(\alpha / \beta + \overline{\alpha/ \beta}\) is invariant under the action of \(D_{4,1}\times C_2\).
Hence \(\alpha/ \beta + \overline{\alpha/ \beta} \in \Q\).
Since \(\alpha/\beta\notin \Q\), we deduce that 
	\begin{equation*}
		(x-\alpha/\beta)(x-\overline{\alpha/\beta}) = x^2 - (\alpha/ \beta + \overline{\alpha /\beta})x + 1 \in\Q[x].
	\end{equation*}
Thus \(\alpha/\beta\) satisfies a quadratic monic polynomial; moreover, since \(\alpha/\beta\) is an algebraic integer it follows that \(\alpha/ \beta + \overline{\alpha/ \beta} \in\Z\).
Since \(|\alpha/\beta|=1\), we find \(\alpha/ \beta + \overline{\alpha/ \beta} \in\{\pm 2, \pm 1, 0\}\).
Each of the roots of the five possible polynomials are roots of unity, which contradicts our assumption on the conjugate ratios.
Mutatis mutandis, one eliminates the possibility that the Galois group is \(D_{4,2}\times C_2\) by similar consideration of \(\alpha/\obeta + \oalpha/\beta\).
Thus the Galois group of the splitting field of \(f\) is either \(S_4\times C_2\) or \(A_4\times C_2\), as required.
 \end{proof}

\begin{remark}
It is not possible to strengthen the above result:
the Galois group of the palindrome \(x^8 + x^7 - x^6 + x^5 + 5x^4 + x^3 - x^2 + x + 1\) is \(S_4\times C_2\), whilst
the Galois group of the palindrome \(x^8 + x^7 - 3x^6 + x^5 + 9x^4 + x^3 - 3x^2 + x + 1\) is \(A_4\times C_2\).
Both polynomials satisfy the assumptions in \autoref{thm:keilthy}.
For the avoidance of doubt, the converse of the statement in \autoref{thm:keilthy} is not true: the Galois group of the palindrome \(x^8 + x^6 + 6x^5 + 9x^4 + 6x^3 + x^2 + 1\), which possesses a single complex-conjugate pair of dominant roots, is \(S_4\times C_2\).
\end{remark}

%% file: future.tex
\section{Directions for Future Research} \label{sec:future}

Our main contribution to the state of the art is a new proof of \autoref{thm:main}: 
the Skolem Problem is decidable for the class of reversible integer linear recurrence sequences of order at most seven.
The benefit of our approach (by comparison to the case analysis in \cite{lipton2021skolem}) is the potential for applications to related decision problems.
In this section we suggest directions for future research; in particular, variations on the Positivity and Skolem Problems for linear recurrence sequences.

\subsection{The Simple Reversible Positivity Problem} \label{ssec:SRPP}
The \emph{Positivity Problem} asks to decide whether the terms in an integer linear recurrence sequence are all non-negative.
Ouaknine and Worrell \cite{ouaknine2014positivity} demonstrated that the Positivity Problem is decidable for \emph{simple} linear recurrence sequences
(those whose characteristic polynomials have no repeated roots) of order at most nine.
The proofs therein very much depend on an approach via Baker's Theorem for linear forms in logarithms.
Those authors identify the class of non-degenerate linear recurrence sequences of order ten that are not amenable to said approach:
the characteristic polynomials in this class have one dominant real root, four complex-conjugate pairs of dominant roots, and one non-dominant root.

Consider the family of monic polynomials in \(\Z[x]\) of degree ten, with constant coefficient \(\pm 1\), and the above distribution of roots.
Suppose that \(f\in\Z[x]\) is a polynomial in this class.
We immediately find that \(f\) is irreducible since each dominant root of \(f\) is necessarily conjugate to the single non-dominant root.
The roots of \(f\) lie on two concentric circles centred at the origin and so we can invoke \autoref{thm:bicycle} to derive a contradiction. %
Thus the obstruction to decidability falls away under our extra assumptions and so we extend the result in \cite{ouaknine2014positivity} in this restricted setting.
In summary,
the \emph{Simple Reversible Positivity Problem} is decidable  for recurrences of order at most ten and so we have proved \autoref{cor:srpp}.

\begin{remark}
Let us give an alternative proof for \autoref{cor:srpp}; this alternative proof does not invoke \autoref{thm:bicycle}.

Observe that any candidate polynomial in our discussion is irreducible.
We can invoke standard results for irreducible polynomials with many dominant roots.
For example, versions of the following lemma are found in \cite{smyth1982conjugate, ferguson1997irreducible}.
	\begin{lemma} \label{lem:manyroots}
	 Suppose that \(\alpha\) is an algebraic number with Galois conjugates \(\beta\) and \(\gamma\) satisfying \(\alpha^{2} = \beta \gamma\).
	 Then the conjugate ratio \(\alpha/\beta\) is a root of unity.
	\end{lemma}

Now suppose that a polynomial \(f\in\Z[x]\) of degree ten has the aforementioned distribution of roots and is the characteristic polynomial of a non-degenerate linear recurrence sequence.
Since \(f\) is necessarily irreducible and has a dominant positive root, we can invoke \autoref{lem:manyroots}.
We deduce that \(f\) is the characteristic polynomial of a degenerate recurrence sequence, a contradiction.
Thus the obstacle to deciding positivity of simple linear recurrence sequences at order ten falls away under our additional assumption of reversibility: there are no sequences in the aforementioned class.
As an aside, \autoref{lem:manyroots} is used by Dubickas and Smyth \cite{dubickas2001remak} as a stepping-stone towards \autoref{thm:bicycle}.

We propose that such approaches as outlined above could lead to further decidability results for variants of the Positivity Problem.
The key observations on the characteristic polynomials were: irreducibility and a dominant positive root.
Here we can make the latter assumption without loss of generality.
Indeed, let us recall the following classical consequence of Pringsheim's Theorem in complex analysis that is pertinent to deciding positivity.
	\begin{lemma} \label{lem:positivity}
	Suppose that a non-zero real-valued linear recurrence sequence \(\seq{X_n}{n}\) has no positive dominant characteristic root. 
	Then the cardinalities of the sets \(\{n\in\N : X_n>0\}\) and \(\{n\in\N : X_n < 0\}\) are both infinite.
	\end{lemma}
\end{remark}

\subsection{The Skolem Problem and unit-norm roots} \label{ssec:SPunitnorm}

	In this paper we invoke results such as \autoref{thm:bicycle} in order to reduce the reversible Skolem Problem at orders five, six, and seven, to decidable instances of the Skolem Problem (\autoref{prop:maxevalues}).
	It is interesting to speculate that techniques involving identities between roots (the Galois theory underlying \autoref{thm:bicycle}) have further applications in establishing decidability results for linear recurrence sequences.
	Indeed, the general version of \autoref{thm:bicycle} (see \cite[Theorem 2.1]{dubickas2001remak}) considers not only algebraic integers that are units, but also algebraic numbers that are \emph{unit-norms}.
	An algebraic number \(\alpha\) is a \emph{unit-norm} if the minimal polynomial of \(\alpha\) is of the form \(a_d x^d - a_{d-1} x^{d-1} - \cdots - a_1 x - a_0 \in\Z[x]\) such that \(|a_d|=|a_0|\).
	So the unit-norm algebraic integers are the units.

	We can strengthen the statement in \autoref{prop:unit5} by invoking \cite[Theorem 2.1]{dubickas2001remak}: we deduce there is no polynomial \(a_5 x^5 - a_4x^4 - a_3x^3 -a_2x^2 - a_1 x \pm a_5\in\Z[x]\) that satisfies hypotheses \ref{h1} and \ref{h2}.
	Thus not only do we settle decidability of the reversible Skolem Problem at order five, but also decidability of the Skolem Problem at order five for the class of rational-valued linear recurrence sequences that satisfy a relation of the form
		\begin{equation*}
		 X_{n+5} = a_{4} X_{n+4} + a_{3} X_{n+3} + a_{2} X_{n+2} + a_1 X_{n+1} \pm X_n
		\end{equation*}
 with \(a_1, a_2, a_3, a_4 \in\Q\).
 Thus we have established \autoref{cor:spunitnorm}.
%